\documentclass[a4paper,12pt]{amsart}

\usepackage{amsfonts}
\usepackage{amssymb}
\usepackage{graphicx,color}
\usepackage{amsmath}
\usepackage{amsthm}

\newtheorem{theorem}{Theorem}[section]
\newtheorem{proposition}{Proposition}[section]
\newtheorem{lemma}{Lemma}[section]

\theoremstyle{definition}

\numberwithin{equation}{section}
\numberwithin{theorem}{section}
\numberwithin{lemma}{section}
\numberwithin{corollary}{section}
\numberwithin{proposition}{section}

\newcommand{\wC}{\widetilde{C}}
\newcommand{\wD}{\widetilde{D}}
\newcommand{\wP}{\widetilde{P}}
\newcommand{\wU}{\widetilde{U}}

\begin{document}

\title[Approximation by Modified Goodman-Sharma Operators]
      {Higher Order Approximation of Functions \\ by Modified Goodman-Sharma Operators}

\author[I. Gadjev]{Ivan Gadjev$^1$}
\address{$^1$\,Faculty of Mathematics and Informatics, Sofia University St. Kliment Ohridski, \newline
         \indent 5, J. Bourchier Blvd, 1164 Sofia, Bulgaria}
\email{gadjev@fmi.uni-sofia.bg}

\author[P. Parvanov]{Parvan Parvanov$^2$}
\address{$^2$\,Faculty of Mathematics and Informatics, Sofia University St. Kliment Ohridski, \newline
         \indent 5, J. Bourchier Blvd, 1164 Sofia, Bulgaria}
\email{pparvan@fmi.uni-sofia.bg}

\author[R. Uluchev]{Rumen Uluchev$^3$}
\address{$^3$\,Faculty of Mathematics and Informatics, Sofia University St. Kliment Ohridski, \newline
         \indent 5, J. Bourchier Blvd, 1164 Sofia, Bulgaria}
\email{rumenu@fmi.uni-sofia.bg}

\thanks{$^{1,2,3}$\,This study is financed by the European Union-NextGenerationEU, through the
        National Recovery and Resilience Plan of the Republic of Bulgaria, project No
        BG-RRP-2.004-0008.}



\date{}

\begin{abstract}
  Here we study the approximation properties of a modified Goodman-Sharma operator recently
  considered by Acu and Agrawal in~\cite{AcAg2019}. This operator is linear but not positive.
  It has the advantage of a higher order of approximation of functions compared with the
  Goodman-Sharma operator. We prove direct and strong converse theorems in terms of a related
  K-functional.
\end{abstract}

\maketitle

%====================================================================================================

\bigskip
\section{Introduction}

In 1987, W. Chen and idenpendently T.\,N.\,T. Goodman and A. Sharma presented at conferences
in China and Bulgaria, respectively a new modification of the classical Bernstein operators.
For $n\in \mathbb{N}$ and functions $f(x)\in C[0,1]$ they introduce the linear operator
(see \cite{Ch1987} and \cite{GoSh1988,GoSh1991}):
\begin{equation} \label{eq:1.1}
  \ U_n(f,x) = f(0)P_{n,0}(x)
               + \sum_{k=1}^{n-1} \!\Big(\int_0^1 \!(n-1)P_{n-2,k-1}(t)f(t)\,dt\Big)P_{n,k}(x)
               + f(1)P_{n,n}(x),
\end{equation}
where
\begin{equation} \label{eq:1.2}
  P_{n,k}(x) = \binom{n}{k}x^k(1-x)^{n-k}, \qquad k=0,\ldots,n.
\end{equation}
Operators of this kind were investigated by many authors (see \cite{PaPo1994}, \cite{BeGoKaTa2002},
\cite{Pal2007}, \cite{IvPa2009}, \cite{GoPa2010-1,GoPa2010-2}, \cite{AcRa2016}, etc.) and are
generally known as genuine Bernstein-Durrmeyer operators. Note that the operators in \eqref{eq:1.1} are
actually a limit case of Bernstein type operators with Jacobi weights studied by Berens and
Xu~\cite{BeXu1991}.

If we set
$$
  u_{n,k}(f) = \begin{cases}
                 f(0), & k=0, \\
                 (n-1)\int_0^1 P_{n-2,k-1}(t)f(t)\,dt, & k=1,\ldots,n-1, \\
                 f(1), & k=n,
               \end{cases}
$$
the operators defined in \eqref{eq:1.1} take the form
$$
  U_n(f,x) = \sum_{k=0}^n u_{n,k}(f)P_{n,k}(x) \qquad\text{or}\qquad
  U_n f = \sum_{k=0}^n u_{n,k}(f)P_{n,k}.
$$

Let us denote, as usual, by
$$
  \varphi(x) = x(1-x)
$$
the weight function which is naturally connected to the second derivative of the Bernstein operator.
Also, we set
\begin{equation} \label{eq:1.3}
  \wD f(x) := \varphi(x)f''(x)
\end{equation}
and$$
  \wD^2 f := \wD \wD f, \qquad \wD^{\ell + 1} f := \wD \wD^{\ell} f, \qquad \ell=2,3\ldots\,.
$$

Recently, Acu and Agrawal \cite{AcAg2019} studied a family of  Bernstein-Durrmeyer operators, as they modify $U_n f$ by replacing the Bernstein basis polynomials $P_{n,k}$ with linear combinations
of  Bernstein basis polynomials of lower degree with coefficients which are polynomials of
appropriate degree. For special choice of the parameters, these operators lack the positivity but
have a higher than $O(n^{-1})$ order of approximation. For example, Acu and Agrawal considered
operators with $O(n^{-2})$ and $O(n^{-3})$ rate of approximation, see \cite[Section~3]{AcAg2019}.

The results presented in \cite{AcAg2019} inspired the authors of the current paper to explore in
more depth the operators explicitly defined by
\begin{equation} \label{eq:1.4}
  \wU_n(f,x) = \sum_{k=0}^n u_{n,k}(f)\wP_{n,k}(x) , \qquad x\in[0,1],
\end{equation}
where
\begin{equation} \label{eq:1.5}
    \wP_{n,k}(x) = P_{n,k}(x) - \frac{1}{n}\,\wD P_{n,k}(x).
\end{equation}
%As a result of our investigations,
By defining an appropriate K-functional, we prove direct and strong converse inequality of type B in  terminology of \cite{DiIv1993}. In order to state our main result, we need some definitions.

Let $L_{\infty}[0,1]$ be the space of all Lebesgue measurable and essentially bounded functions in
$[0,1]$ and $AC_{loc}(0,1)$ consists of the functions absolutely continuous in any subinterval
$[a,b]\subset(0,1)$. Let us set
$$
  W^2(\varphi)[0,1]  := \big\{g\,:\,g,g'\in AC_{loc}(0,1), \ \wD g\in L_{\infty}[0,1]\big\}.
$$
By $W^2_0(\varphi)[0,1]$  we denote the subspace of $W^2(\varphi)[0,1]$  of functions $g$ satisfying the additional boundary conditions

\begin{equation*}
    \lim_{x\to 0^{+}}\wD g = 0, \qquad \lim_{x\to 1^{-}}\wD g = 0.
\end{equation*}

Henceforth, by $\|\cdot\|$ we mean the uniform norm on the interval $[0,1]$. For functions
$f\in C[0,1]$ and $t>0$ we define the K-functional
\begin{equation} \label{eq:1.6}
  K(f,t) := \inf \big\{\|f-g\|+t\|\wD^2 g\|\,:\,g\,\in W^2_0(\varphi)[0,1],\,\wD g\in W^2(\varphi)[0,1]\big\}.
\end{equation}

Here we investigate the error of approximation of  functions $f\in C[0,1]$ by the modified
Goodman-Sharma operator \eqref{eq:1.4}. Our main results read as follows.

\begin{theorem} \label{th:1.1}
  If $\,n\in\mathbb{N}$, $n\ge 2$, and $f\in C[0,1]$, then
  $$
    \big\|\wU_n f-f\big\| \le (1+\sqrt{3})\,K\Big(f,\frac{1}{n^2}\Big).
  $$
\end{theorem}

\begin{theorem} \label{th:1.2}
  For every function $f\in C[0,1]$ and $n\in\mathbb{N}$, $n\ge 2$, there exist constants $C,L>0$
  such that
  $$
    K\Big(f,\frac{1}{n^2}\Big) \le
    C\,\frac{\ell^2}{n^2}\big(\big\|\wU_n f - f \big\| + \big\|\wU_{\ell} f - f) \big\| \big).
  $$
  for all $\ell\ge Ln$.
\end{theorem}

\textit{Remark}. Another way to state Theorem~\ref{th:1.1} and Theorem~\ref{th:1.2} is: there exists a natural number $k$ such that
$$
  K\Big(f,\frac{1}{n^2}\Big) \sim \big\|\wU_n f - f \big\| + \big\|\wU_{kn} f - f \big\|.
$$

The paper is organized as follows. In Section~1 state of the art is described. Preliminary and
auxiliary results are presented in Section~2. Section~3 includes an estimation of the norm of
the operator $\wU_n$, a Jackson type inequality and a proof of the direct inequality in
Theorem~\ref{th:1.1}. The last Section~4 is devoted to a converse result for the modified
Goodman-Sharma operator \eqref{eq:1.4}. Inequalities of the Voronovskaya type and Bernstein type
for $\wU_n$ are proved using the differential operator~$\wD$, defined in \eqref{eq:1.3}.
Theorem~\ref{th:1.2} represents a strong converse inequality of Type B, according to Ditzian-Ivanov
classification in \cite{DiIv1993}. Complete proof of the converse theorem is given.

%====================================================================================================

\smallskip
\section{Preliminaries and Auxiliary Results}

By $B_n f$, $n\in \mathbb{N}$, we denote the Bernstein operators determined for functions $f$,
$$
  B_n(f,x) = \sum_{k=0}^n f\Big(\frac{k}{n}\Big)P_{n,k}(x), \qquad x\in[0,1],
$$
where $P_{n,k}$ are the Bernstein basis polynomials \eqref{eq:1.2}. The Bernstein operators central
moments play important role in many applications and they are defined by
$$
  \mu_{n,i}(x) = B_n\big((t-x)^i,x\big) = \sum_{k=0}^n \Big(\frac{k}{n}-x\Big)^i P_{n,k}(x), \qquad
  i=0,1,\ldots\,.
$$

We summarize some well known useful properties of the Bernstein polynomials. Further on we assume
$P_{n,k}:=0$ if $k<0$ or $k>n$.

\begin{proposition}[see, e.g. \cite{Lo1953}] \label{pr:2.1}
  (a) The following identities are valid:
      \begin{align} \label{eq:2.1}
        & \sum_{k=0}^n k P_{n,k}(x) = nx, \qquad \sum_{k=0}^n (n-k)P_{n,k}(x) = n(1-x), \\
        & \sum_{k=0}^n k(k-1)P_{n,k}(x) = n(n-1)x^2, \label{eq:2.2} \\
        & \sum_{k=0}^n (n-k)(n-k-1)P_{n,k}(x) = n(n-1)(1-x)^2, \label{eq:2.3} \\
        & P_{n,k}'(x) = n\big[P_{n-1,k-1}(x)-P_{n-1,k}(x)\big], \label{eq:2.4} \\
        & P_{n,k}''(x) = n(n-1)\big[P_{n-2,k-2}(x)-2P_{n-2,k-1}(x)+P_{n-2,k}(x)\big]. \label{eq:2.5}
      \end{align}

  (b) For the low-order moments $\mu_{n,i}(x)$ we have:
      \begin{align*}
        \mu_{n,0}(x) & = B_n\big((t-x)^0,x\big) = 1, \\
        \mu_{n,1}(x) & = B_n\big((t-x),x\big) = 0, \\
        \mu_{n,2}(x) & = B_n\big((t-x)^2,x\big) = \frac{\varphi(x)}{n}, \\
        \mu_{n,3}(x) & = B_n\big((t-x)^3,x\big) = \frac{(1-2x)\varphi(x)}{n^2}, \\
        \mu_{n,4}(x) & = B_n\big((t-x)^4,x\big) = \frac{3(n-2)\varphi^2(x)}{n^3} + \frac{\varphi(x)}{n^3}.
      \end{align*}
\end{proposition}

The operators $U_n$, $\wU_n$ and the differential operator $\wD$ satisfy interesting properties.

\begin{proposition} \label{pr:2.2}
  If the operators $U_n$, $\wU_n$ and the differential operator $\wD$ are defined as in
  \eqref{eq:1.1}, \eqref{eq:1.4} and \eqref{eq:1.3}, respectively, then

  \smallskip
  (a) \ $\wD U_n f = U_n \wD f$ for   $f\in W^2_0(\varphi)[0,1]$;

  \smallskip
  (b) \ $\wU_n f = U_n\big(f-\frac{1}{n}\,\wD f\big)$ for   $f\in W^2_0(\varphi)[0,1]$;

  \smallskip
  (c) \ $\wD \wU_n f = \wU_n \wD f$ for   $f\in W^2_0(\varphi)[0,1]$;

  \smallskip
  (d) \ $U_n \wU_n f = \wU_n U_n f$ for   $f\in W^2_0(\varphi)[0,1]$;

  \smallskip
  (e) \ $\wU_m \wU_n f = \wU_n \wU_m f$ for   $f\in W^2_0(\varphi)[0,1]$;

  \smallskip
  (f) \ $\displaystyle \lim_{n\to \infty} \wU_n f = f$ for   $f\in W^2(\varphi)[0,1]$;

  \smallskip
  (g) \ $\displaystyle  \big\| \wD U_n f\big\|  \le \big\|\wD f\big\|$  for   $f\in W^2(\varphi)[0,1]$.
\end{proposition}

\begin{proof}
For the proof of (a), see \cite[Lemma~4.2]{PaPo1994}.

We have
\begin{align*}
  \wU_{n} f & = \sum_{k=0}^n u_{n,k}(f)\wP_{n,k} \\
            & = u_{n,0}(f)\Big(P_{n,0} - \frac{1}{n}\,\wD P_{n,0}\Big)
                + \sum_{k=1}^{n-1} u_{n,k}(f)\Big(P_{n,k} - \frac{1}{n}\,\wD P_{n,k}\Big) \\
            & \quad + u_{n,n}(f)\Big(P_{n,n} - \frac{1}{n}\,\wD P_{n,n}\Big) \\
            & = u_{n,0}(f)P_{n,0} + \sum_{k=1}^{n-1} u_{n,k}(f) P_{n,k} + u_{n,n}(f)P_{n,n} \\
            & \quad -\frac{\varphi}{n}\bigg(u_{n,0}(f)P_{n,0}'' + \sum_{k=1}^{n-1} u_{n,k}(f)P_{n,k}'' + u_{n,n}(f)P_{n,n}''\bigg) \\
            & = U_n f - \frac{1}{n}\,\varphi U_n'' f.
\end{align*}
Then from (a) we obtain
$$
  \wU_n f = U_n f - \frac{1}{n}\,\wD U_n f = U_n f - \frac{1}{n}\,U_n \wD f
          = U_n\Big(f-\frac{1}{n}\,\wD f\Big)
$$
which proves (b).

Now, commutative properties (c) and (d) follow from (b) and (a):
$$
  \wD \wU_n f = \wD U_n\Big(f-\frac{1}{n}\,\wD f\Big) = U_n\Big(\wD f - \frac{1}{n}\,\wD \wD f\Big)
               = \wU_n(\wD f),
$$
and
\begin{align*}
  U_n \wU_n f & = U_n U_n\Big(f-\frac{1}{n}\,\wD f\Big) = U_n U_n f - \frac{1}{n}\,U_n U_n \wD f \\
                & = U_n U_n f - \frac{1}{n}\,U_n \wD U_n f = \wU_n U_n f.
\end{align*}

The operators $\wU_n$ commute in the sense of (e), since
\begin{align*}
  \wU_m \wU_n f & = \wU_mU_n\big(f-\frac{1}{n}\,\wD f\big)\\
                & = U_m U_n f - \frac{1}{n}\,U_m U_n \wD f -\frac{1}{m}\,\wD U_m U_n f + \frac{1}{mn}\,U_m \wD^2 U_n f\\
%               & = U_m U_n f - \frac{1}{n}\,\wD U_m U_n  f -\frac{1}{m}\,\wD U_m U_n f + \frac{1}{mn}\,\wD^2 U_m  U_n f. \\
                & = U_m U_n\Big(f-\frac{m+n}{mn}\,\wD f + \frac{1}{mn}\,\wD^2 f\Big).
\end{align*}
The same expression on the right-hand side we obtain for $\wU_n \wU_m f$ because of properties
(a), (b) and $U_m U_n f = U_n U_m f$.

We recall two more properties of the operator $U_n$ and function  $f\in W^2(\varphi)[0,1]$ (see \cite[eqs.~(4.8), (2.4)]{PaPo1994}):
\begin{align} \label{eq:2.6}
  \|U_n f  - f\| & \le \frac{1}{n}\,\big\|\wD f\big\|, \\
  \big\|U_n \wD f\big\| & \le \big\|\wD f\big\|. \notag
\end{align}
Therefore
$$
  \|\wU_n f - f\| = \Big\|U_n f  -\frac{1}{n}\,U_n \wD f - f\Big\|
                   \le \|U_n f - f\| + \frac{1}{n}\,\big\|U_n \wD f\big\|
                   \le \frac{2}{n}\,\|\wD f\|,
$$
hence $\lim\limits_{n\to \infty} \|\wU_n f - f\|=0$, i.e. the limit (f) holds true.

From the proof of Lemma 4.2 in \cite {PaPo1994} for every $g\in W^2(\varphi)[0,1]$ we have
\begin{equation*}
  \wD U_ng(x) = \sum_{k=1}^{n-1} P_{n,k}(x)\int_0^1 (n-1) P_{n-2,k-1}(t)\wD g(t)\,dt,
\end{equation*}
From the above representation we obtain
$$
  |\wD U_ng(x)|
  \le \big\| \wD g \big\|\,\sum _{k=1}^{n-1}P_{n,k}(x)\int _0^1 (n-1) P_{n-2,k-1}(t)\,dt
  \le \big\| \wD g \big\|,
$$
which proves (g).
\end{proof}

\smallskip
We now introduce a function that will prove useful in our investigations:
\begin{align} \label{eq:2.7}
  T_{n,k}(x) & := k(k-1)\frac{1-x}{x} - 2k(n-k) + (n-k)(n-k-1)\frac{x}{1-x} \\
             & = n\bigg[-1-\frac{1-2x}{\varphi(x)}\Big(\frac{k}{n}-x\Big) + \frac{n}{\varphi(x)}\Big(\frac{k}{n}-x\Big)^2\bigg].
               \notag
\end{align}

Observe that
\begin{align}
  T_{n,k}'(x)  & = - \frac{k(k-1)}{x^2} + \frac{(n-k)(n-k-1)}{(1-x)^2}, \label{eq:2.8} \\
  T_{n,k}''(x) & = \frac{2k(k-1)}{x^3} + \frac{2(n-k)(n-k-1)}{(1-x)^3} > 0, \qquad x\in (0,1). \label{eq:2.9}
\end{align}

\begin{proposition} \label{pr:2.3}
(a) The following relation concerning $P_{n,k}$, $T_{n,k}$ and differential operator $\wD$ holds:
\begin{equation} \label{eq:2.10}
  \wD P_{n,k}(x) = T_{n,k}(x) P_{n,k}(x).
\end{equation}

(b) If $\,\alpha$ is an arbitrary real number, then
$$
  \Phi(\alpha) := \sum_{k=0}^n \Big(\alpha - \frac{1}{n}\,T_{n,k}(x)\Big)^2 P_{n,k}(x) = \alpha^2 + 2 - \frac{2}{n}.
$$
\end{proposition}

\begin{proof}
(a) \ From \eqref{eq:2.4}, \eqref{eq:2.5} and
$$
  \varphi(x)P_{n,k}(x) = \frac{(k+1)(n-k+1)}{(n+1)(n+2)}\,P_{n+2,k+1}(x),
$$
it follows that
\begin{align*}
  \varphi(x)P_{n,k}''(x)
    & = n(n-1)\big[\varphi(x)P_{n-2,k-2}(x)-2\varphi(x)P_{n-2,k-1}(x)+\varphi(x)P_{n-2,k}(x)\big] \\
    & = n(n-1)\Big[\frac{(k-1)(n-k+1)}{n(n-1)}\,P_{n,k-1}(x)-2\,\frac{k(n-k)}{n(n-1)}\,P_{n,k}(x) \\
    & \hspace*{56mm} + \frac{(k+1)(n-k-1)}{n(n-1)}\,P_{n,k+1}(x)\Big] \\
    & = (k-1)(n-k+1)\,P_{n,k-1}(x) - 2k(n-k)\,P_{n,k}(x) \\
    & \quad + (k+1)(n-k-1)\,P_{n,k+1}(x) \\
    & = \Big[k(k-1)\frac{1-x}{x} - 2k(n-k) + (n-k)(n-k-1)\frac{x}{1-x}\Big]P_{n,k}(x) \\
    & = T_{n,k}(x)P_{n,k}(x),
\end{align*}
i.e. the identity \eqref{eq:2.10}.

(b) \ We apply the formulae for the Bernstein operator moments in Proposition~\ref{pr:2.1}\,(b):
\begin{align*}
  \Phi(\alpha)
    & = \sum_{k=0}^n \Big[\alpha + 1 + \frac{1-2x}{\varphi(x)} \Big(\frac{k}{n}-x\Big)
                          - \frac{n}{\varphi(x)}\Big(\frac{k}{n}-x\Big)^{\!2}\Big]^2 P_{n,k}(x) \\
    & = \sum_{k=0}^n \Big[(\alpha+1)^2 + \frac{(1-2x)^2}{\varphi^2(x)}\Big(\frac{k}{n}-x\Big)^2
                            + \frac{n^2}{\varphi^2(x)}\Big(\frac{k}{n}-x\Big)^4
                            + \frac{2(\alpha+1)(1-2x)}{\varphi(x)}\Big(\frac{k}{n}-x\Big) \\
    & \hspace*{16mm}        - \frac{2(\alpha+1)n}{\varphi(x)}\Big(\frac{k}{n}-x\Big)^2
                            - \frac{2n(1-2x)}{\varphi^2(x)}\Big(\frac{k}{n}-x\Big)^3\Big] P_{n,k}(x) \\
    & = (\alpha+1)^2\mu_{n,0}(x) + \frac{(1-2x)^2}{\varphi^2(x)}\,\mu_{n,2}(x) + \frac{n^2}{\varphi^2(x)}\,\mu_{n,4}(x)
                      + \frac{2(\alpha+1)(1-2x)}{\varphi(x)}\,\mu_{n,1}(x) \\
    & \hspace*{20mm}  - \frac{2(\alpha+1)n}{\varphi(x)}\,\mu_{n,2}(x) - \frac{2n(1-2x)}{\varphi^2(x)}\,\mu_{n,3}(x) \\
    & = (\alpha+1)^2\cdot 1 + \frac{(1-2x)^2}{\varphi^2(x)}\,\frac{\varphi(x)}{n}
                            + \frac{n^2}{\varphi^2(x)}\,\frac{(3n-6)\varphi^2(x)+\varphi(x)}{n^3} \\
    & \hspace*{20mm}        + \frac{2(\alpha+1)(1-2x)}{\varphi(x)}\cdot 0 - \frac{2(\alpha+1)n}{\varphi(x)}\,\frac{\varphi(x)}{n}
                            - \frac{2n(1-2x)}{\varphi^2(x)}\,\frac{(1-2x)\varphi(x)}{n^2} \\
    & = (\alpha+1)^2 + \frac{1-4\varphi(x)}{n\varphi(x)} + \frac{(3n-6)\varphi(x)+1}{n\varphi(x)} - 2(\alpha+1)
                     - \frac{2(1-4\varphi(x))}{n\varphi(x)} \\
    & = \alpha^2 + 2\alpha + 1 + \frac{1}{n\varphi(x)} - \frac{4}{n} + 3 - \frac{6}{n} + \frac{1}{n\varphi(x)}
                               -2\alpha - 2 - \frac{2}{n\varphi(x)} + \frac{8}{n} \\
   & = \alpha^2 + 2 - \frac{2}{n}.
\end{align*}
\end{proof}

Auxiliary technical results will be useful for further estimations.

\begin{proposition} \label{pr:2.4}
  If $n\in \mathbb{N}$, $n\ge 2$, and
  $$
    \lambda(n) := \sum_{k=n}^{\infty} \frac{1}{k^2(k+1)}, \qquad
    \theta(n) := \sum_{k=n}^{\infty} \frac{1}{k^2(k+1)^2},
  $$
  then
  \begin{gather} \label{eq:2.11}
    \frac{1}{2n^2} \le \lambda(n) \le \frac{1}{n^2}, \\
    \theta(n) \le \frac{4}{9n^3}. \label{eq:2.12}
  \end{gather}
\end{proposition}

\begin{proof}
Since $\frac{k}{k-1}\frac{n-1}n\le 1$ for $k\ge n$, we have for
the lower estimate of $\lambda(n)$
$$
   \lambda(n) \ge \sum_{k=n}^{\infty} \frac{1}{k^2(k+1)}.\frac{k}{k-1}.\frac{n-1}n =\frac{n-1}n \sum_{k=n}^{\infty} \frac{1}{(k-1)k(k+1)}=\frac{n-1}n.\frac1{2(n-1)n}=  \frac{1}{2n^2}.
$$
For the upper estimates of $\lambda(n)$ and $\theta(n)$,  we obtain
\begin{gather*}
  \lambda(n)<\sum_{k=n}^{\infty} \frac{1}{(k-1)k(k+1)} = \frac{1}{2n(n-1)} < \frac{1}{n^2}, \\
  \theta(n)<\sum_{k=n}^{\infty} \frac{1}{(k-1)k(k+1)(k+2)} = \frac{1}{3n(n^2-1)} < \frac{4}{9n^3}.
\end{gather*}
\end{proof}

%====================================================================================================

\smallskip
\section{A Direct Theorem}

We will first prove the next upper estimate for the norm of the operator $\wU_n$ defined in \eqref{eq:1.4}.

\begin{lemma} \label{le:3.1}
  If $\,n\in \mathbb{N}$ and $f\in C[0,1]$, then
  \begin{equation} \label{eq:3.1}
    \big\|\wU_n f\big\| \le \sqrt{3}\,\|f\|, \quad\mbox{i.e.} \quad  \|\wU_n\| \le \sqrt{3}.
  \end{equation}
\end{lemma}

\begin{proof}
We have
$$
  \wP_{n,k}(x) = P_{n,k}(x) - \frac{1}{n}\,\wD P_{n,k}(x) = \Big(1-\frac{1}{n}\,T_{n,k}(x)\Big)P_{n,k}(x).
$$
Then for $x\in[0,1]$,
\begin{align*}
  \big|\wU_n(f,x)\big|
  & = \left| \sum_{k=0}^n u_{n,k}(f)\wP_{n,k}(x) \right|
    \le \sum_{k=0}^n |u_{n,k}(f)|\,\big|\wP_{n,k}(x)\big| \\
  & \le \|f\| \sum_{k=0}^n \big|\wP_{n,k}(x)\big|
    = \|f\| \sum_{k=0}^n \Big|1-\frac{1}{n}\,T_{n,k}(x)\Big| P_{n,k}(x).
\end{align*}

Applying Cauchy inequality we obtain
$$
  \big|\wU_n(f,x)\big| \le
  \|f\| \sqrt{\sum_{k=0}^n \Big(1-\frac{1}{n}\,T_{n,k}(x)\Big)^2 P_{n,k}(x)}\,\sqrt{\sum_{k=0}^n P_{n,k}(x)}.
$$
Since $\sum_{k=0}^n P_{n,k}(x)=1$ identically, by Proposition~\ref{pr:2.3}\,(b) with $\alpha=1$ we find
$$
  \big|\wU_n(f,x)\big| \le \sqrt{3-\frac{2}{n}}\,\|f\| < \sqrt{3}\,\|f\|, \qquad x\in [0,1].
$$
Hence, inequality \eqref{eq:3.1} follows.
\end{proof}

\smallskip
In order to prove a direct theorem for the approximation rate for functions $f$ by the operator
$\wU_n f$ we need a Jackson type inequality.

\begin{lemma} \label{le:3.2}
  If $\,n\in \mathbb{N}$, $f\in W^2_0(\varphi)[0,1]$ and $\wD f\in W^2(\varphi)[0,1]$, then
  \begin{equation} \label{eq:3.2}
    \big\|\wU_n f - f\big\| \le \frac{1}{n^2}\|\wD^2 f\|.
  \end{equation}
\end{lemma}

\begin{proof}
Having in mind the relation
$$
  U_k f - U_{k+1} f = \frac{1}{k(k+1)}\,\wD U_{k+1} f,
$$
(see \cite[Lemma~4.1]{PaPo1994}) and Proposition~\ref{pr:2.1}\,(a) for $f\, \in W^2_0(\varphi)[0,1]$, we obtain
\begin{align*}
  \wU_k f - \wU_{k+1} f
    & = U_k f - \frac{1}{k}\,\wD U_k f - U_{k+1} f + \frac{1}{k+1}\,\wD U_{k+1} f \\
    & = U_k f - U_{k+1} f + \frac{1}{k+1}\,\wD U_{k+1} f - \frac{1}{k}\,\wD U_k f \\
    & = \Big(\frac{1}{k}-\frac{1}{k+1}\Big)\wD U_{k+1} f + \frac{1}{k+1}\,\wD U_{k+1} f - \frac{1}{k}\,\wD U_k f \\
    & = - \frac{1}{k}\big(\wD U_k f - \wD U_{k+1} f\big) \\
    & = - \frac{1}{k}\big(U_k \wD f - U_{k+1} \wD f\big) \\
    & = - \frac{1}{k}\cdot\frac{1}{k(k+1)}\, \wD U_{k+1} \wD f, \\
  \end{align*}
i.e.
\begin{equation} \label{eq:3.3}
  \wU_k f - \wU_{k+1} f = - \frac{1}{k^2(k+1)}\,\wD U_{k+1} \wD f.
\end{equation}

  Therefore for every $s>n$ we have
\begin{gather*}
 \wU_n f - \wU_s f=  \sum_{k=n}^{s-1} \big(\wU_k f - \wU_{k+1} f\big)
    = - \sum_{k=n}^{s-1} \frac{1}{k^2(k+1)}\,\wD U_{k+1} \wD f.
   \end{gather*}
   Letting $s\rightarrow\infty$ and by  Proposition~\ref{pr:2.2}\,(a) and (f) we obtain
\begin{equation}\label{eq1}
  \wU_n f - f = - \sum_{k=n}^{\infty} \frac{1}{k^2(k+1)}\,\wD U_{k+1} \wD f
\end{equation}

Then  from Proposition~\ref{pr:2.1}\,(g) for $ \wD f\in W^2(\varphi)[0,1]$
$$
  \|\wU_n f - f\| \le \sum_{k=n}^{\infty} \frac{1}{k^2(k+1)}\,\big\|\wD U_{k+1} \wD f\big\|
                  \le \sum_{k=n}^{\infty} \frac{1}{k^2(k+1)}\,\big\|\wD^2 f\big\|.
$$
Proposition~\ref{pr:2.4}, \eqref{eq:2.11}, yields
$$
  \sum_{k=n}^{\infty} \frac{1}{k^2(k+1)} \le \frac{1}{n^2}.
$$

Therefore
$$
  \big\|\wU_n f - f\big\| \le \frac{1}{n^2}\,\big\|\wD^2 f\big\|.
$$
\end{proof}

A direct result on the approximation rate of functions $f\in C[0,1]$ by the operators
\eqref{eq:1.4} in means of the K-functional \eqref{eq:1.6} follows immediately from both
lemmas above.

\smallskip
\begin{proof}[Proof of Theorem~\ref{th:1.1}]
Let $g$ be arbitrary function, such that $g\in W^2_0(\varphi)[0,1]$ and $\wD g\in W^2(\varphi)[0,1]$. Then by Lemma~\ref{le:3.1}
and Lemma~\ref{le:3.2} we obtain
\begin{align*}
  \big\|\wU_n f-f\big\|
    & \le \big\|\wU_n f-\wU_n g\big\| + \big\|\wU_n g-g\big\| + \|g-f\| \\
    & \le (1+\sqrt{3})\|f-g\| + \frac{1}{n^2}\,\big\|\wD^2 g\big\| \\
    & \le (1+\sqrt{3})\Big(\|f-g\|+\frac{1}{n^2}\,\big\|\wD^2 g\big\|\Big).
\end{align*}
Taking infimum over all functions $g$ with $g\in W^2_0(\varphi)[0,1]$ and $\wD g\in W^2(\varphi)[0,1]$ we obtain
$$
  \big\|\wU_n f-f\big\| \le (1+\sqrt{3})\,K\Big(f,\frac{1}{n^2}\Big).
$$
\end{proof}

%====================================================================================================

\smallskip
\section{A Strong Converse Result}

First, we will prove a Voronovskaya type result for the operator $\wU_n$.

\begin{lemma} \label{le:4.1}
  If $\,\lambda(n) = \sum_{k=n}^{\infty} \frac{1}{k^2(k+1)}$,
  $\,\theta(n) = \sum_{k=n}^{\infty} \frac{1}{k^2(k+1)^2}$ and $f\in C[0,1]$ is such that
  $f,\,\wD f\in W^2_0(\varphi)[0,1]$ and $\wD^3 f\in L_{\infty}[0,1]$, then
  \begin{equation} \label{eq:4.1}
    \big\|\wU_n f - f + \lambda(n)\wD^2 f\big\| \le \theta(n)\,\big\|\wD^3 f\big\|.
  \end{equation}
\end{lemma}

\begin{proof}
We have
$$
  \wU_n f - f + \lambda(n)\wD^2 f
   = - \sum_{k=n}^{\infty} \frac{U_{k+1} \wD^2 f}{k^2(k+1)} + \sum_{k=n}^{\infty} \frac{\wD^2 f}{k^2(k+1)}
   = \sum_{k=n}^{\infty} \frac{\wD^2 f - U_{k+1} \wD^2 f}{k^2(k+1)},
$$
see the proof of Lemma~\ref{le:3.2}, eq. \eqref{eq:3.3}. Then
$$
  \big\|\wU_n f - f + \lambda(n)\wD^2 f\big\|
  \le \sum_{k=n}^{\infty} \frac{1}{k^2(k+1)}\,\big\|\wD^2 f - U_{k+1} \wD^2 f\big\|.
$$
Using \eqref{eq:2.6} with $\wD^2 f$ instead of $f$ we obtain
$$
  \big\|\wU_n f - f + \lambda(n)\wD^2 f\big\|
    \le \sum_{k=n}^{\infty} \frac{1}{k^2(k+1)}\cdot\frac{1}{(k+1)}\,\big\|\wD \wD^2 f\big\|
    = \theta(n)\,\big\|\wD^3 f\big\|.
$$
\end{proof}

We need an inequality of Bernstein type.

\begin{lemma} \label{le:4.2}
  Let $n\in\mathbb{N}$, $n\ge 2$ and $f\in C[0,1]$. Then the following inequality holds true
  \begin{equation} \label{bieq:4.1}
    \|\wD \wU_n f\| \le \wC\, n\|f\|,
  \end{equation}
  where $\wC=6.5+\sqrt{6} $.
\end{lemma}

\begin{proof}
Since
$$
  \big|\wD\wU_n(f,x)\big| \le \sum_{k=0}^n |u_{n,k}(f)|\,\big|\wD\wP_{n,k}(x)\big|
                              \le \|f\| \sum_{k=0}^n \big|\wD\wP_{n,k}(x)\big|,
$$
it is sufficient to find an upper estimate for the quantity
$$
  \sum_{k=0}^n \big|\wD \wP_{n,k}(x)\big| = \sum_{k=0}^n \big|\varphi(x)\wP_{n,k}''(x)\big|.
$$

Remind that, according to \eqref{eq:2.10}, we have the relation
$$
  \wD P_{n,k}(x) = \varphi(x)P_{n,k}''(x) = T_{n,k}(x)P_{n,k}(x).
$$
Hence
\begin{gather*}
  \wP_{n,k}(x) = P_{n,k}(x) - \frac{1}{n}\,\wD P_{n,k}(x)
               = \Big(1 - \frac{1}{n}\,T_{n,k}(x)\!\Big)P_{n,k}(x), \\
  \wP_{n,k}''(x) = \Big(1 - \frac{1}{n}\,T_{n,k}(x)\!\Big)'' P_{n,k}(x)
                   + 2\Big(1 - \frac{1}{n}\,T_{n,k}(x)\!\Big)' P_{n,k}'(x)
                   + \Big(1 - \frac{1}{n}\,T_{n,k}(x)\!\Big) P_{n,k}''(x),
\end{gather*}
Then,
\begin{align*}
  \wD \wP_{n,k}(x) & = \varphi(x)\wP_{n,k}''(x) \\
                   & = - \frac{\varphi(x)}{n}\,T_{n,k}''(x) P_{n,k}(x)
                       - \frac{2\varphi(x)}{n}\,T_{n,k}'(x) P_{n,k}'(x)
                       + \Big(1 - \frac{1}{n}\,T_{n,k}(x)\!\Big) \varphi(x) P_{n,k}''(x) \\
                   & = - \frac{\varphi(x)}{n}\,T_{n,k}''(x) P_{n,k}(x)
                       - \frac{2\varphi(x)}{n}\,T_{n,k}'(x) P_{n,k}'(x)
                       + \!\Big(1 - \frac{1}{n}\,T_{n,k}(x)\!\Big) T_{n,k}(x) P_{n,k}(x).
\end{align*}
Therefore
\begin{align*}
  \sum_{k=0}^n \big|\wD \wP_{n,k}(x)\big| & \le a_n(x) + b_n(x) + c_n(x), \\
  \intertext{where}
  a_n(x) & = \frac{\varphi(x)}{n} \sum_{k=0}^n \big|T_{n,k}''(x)\big| P_{n,k}(x), \\
  b_n(x) & = \frac{2\varphi(x)}{n} \sum_{k=0}^n \big|T_{n,k}'(x) P_{n,k}'(x)\big|, \\
  c_n(x) & = \sum_{k=0}^n \Big|\Big(1 - \frac{1}{n}\,T_{n,k}(x)\Big) T_{n,k}(x)\Big| P_{n,k}(x).
\end{align*}

\smallskip
\underline{\sl 1. Estimate for $a_n(x)$.} From \eqref{eq:2.9} and \eqref{eq:2.2}--\eqref{eq:2.3},
\begin{align*}
  \sum_{k=0}^n T_{n,k}''(x) P_{n,k}(x)
    & = \sum_{k=0}^n \Big(\frac{2k(k-1)}{x^3} + \frac{2(n-k)(n-k-1)}{(1-x)^3}\Big) P_{n,k}(x) \\
    & = \frac{2}{x^3} \sum_{k=0}^n k(k-1)P_{n,k}(x) + \frac{2}{(1-x)^3} \sum_{k=0}^n (n-k)(n-k-1)P_{n,k}(x) \\
    & = \frac{2}{x^3}\,n(n-1)x^2 + \frac{2}{(1-x)^3}\,n(n-1)(1-x)^2 \\
    & = \frac{2n(n-1)}{\varphi(x)}.
\end{align*}
Having in mind $T_{n,k}''(x)>0$ in \eqref{eq:2.9}, we obtain
\begin{equation} \label{eq:4.3}
  a_n(x) = \frac{\varphi(x)}{n}\sum_{k=0}^n T_{n,k}''(x) P_{n,k}(x) = 2(n-1).
\end{equation}

\smallskip
\underline{\sl 2. Estimate for $b_n(x)$.} \ Observe that
$$
  \sum_{k=0}^n \big|T_{n,k}'(x)\,P_{n,k}'(x)\big| = \sum_{k=0}^n \big|T_{n,k}'(1-x)\,P_{n,k}'(1-x)\big|,
$$
hence, there is a symmetry of the function $b_n(x)$ in $x=\frac12$. Therefore, it is sufficient to
estimate $b_n(x)$ for $x\in\big[0,\frac12\big]$.

We will show that in $\big[0,\frac12\big]$ the function $b_n(x)$ has exactly
$\big\lfloor \frac{n-1}{2}\big\rfloor$ local extrema $h_k$ attained at points in intervals
$\big(\frac{k-1}{n},\frac{k}{n}\big]$, $k=1,\ldots,\big\lfloor \frac{n-1}{2}\big\rfloor$,
respectively. We will estimate all the local maxima $h_k$ and then an estimate for $b_n(x)$
will follow immediately.

\smallskip
\begin{enumerate}
  \item[(i)]   First, we prove that
               $$
                 S(x) := \frac{-2\varphi(x)}{n}\sum_{k=0}^n T_{n,k}'(x) P_{n,k}'(x) = 4(n-1).
               $$

               From \eqref{eq:2.4} and \eqref{eq:2.8},
               $$
                 \sum_{k=0}^n T_{n,k}'(x) P_{n,k}'(x) =
                 n\sum_{k=0}^{n-1} \big(T_{n,k+1}'(x)-T_{n,k}'(x)\big) P_{n-1,k}(x).
               $$
               Since
               \begin{align*}
                 T_{n,k+1}'(x) - T_{n,k}'(x)
                   & \!=\! \frac{(n-k-1)(n-k-2)}{(1-x)^2} - \frac{(k+1)k}{x^2}
                           + \frac{k(k-1)}{x^2} - \frac{(n-k)(n-k-1)}{(1-x)^2} \\
                   & \!=\! - \frac{2k}{x^2} - \frac{2(n-k-1)}{(1-x)^2},
               \end{align*}
               using \eqref{eq:2.1} we get
               \begin{align*}
                 \sum_{k=0}^{n} T_{n,k}'(x) P_{n,k}'(x)
                   & = - \frac{2n}{x^2} \sum_{k=0}^{n-1} k P_{n-1,k}(x)
                       - \frac{2n}{(1-x)^2} \sum_{k=0}^{n-1} (n-k-1) P_{n-1,k}(x) \\
                   & = - \frac{2n}{x^2}\,(n-1)x - \frac{2n}{(1-x)^2}\,(n-1)(1-x) \\
                   & = - \frac{2n(n-1)}{\varphi(x)}.
               \end{align*}
               Therefore,
               \begin{equation} \label{eq:4.4}
                 S(x) = \frac{-2\varphi(x)}{n}\cdot\frac{-2n(n-1)}{\varphi(x)} = 4(n-1).
               \end{equation}

  \smallskip
  \item[(ii)]  By \eqref{eq:2.9}, $T_{n,k}''(x)>0$, hence $-T_{n,k}'(x)$ strictly decreases for
               $x\in (0,1)$.

               \smallskip
               For $k=0,1$ we have $-T_{n,k}'(0^+)<0$, then $-T_{n,k}'(x)<0$, $x\in (0,1)$, and
               $\varphi(x)T_{n,1}'(x)$ has its only zero in $[0,1)$ at $\xi_1=0$.

               \smallskip
               For $k=2,\ldots,n-2$, we have $-T_{n,k}'(0^+)>0$, and $T_{n,k}'(x)$ has a unique
               simple zero at
               $$
                 \xi_k = \tfrac{\sqrt{\binom{k}{2}}}{\sqrt{\binom{k}{2}}+\sqrt{\binom{n-k}{2}}}
                 \in \big(\tfrac{k-1}{n},\tfrac{k}{n}\big).
               $$

               For $k=n-1,n$, we have $-T_{n,k}'(x)>0$ for $x\in (0,1)$, and
               $-\varphi(x)T_{n,n}'(x)=0$ only for $\xi_n=1$ in $(0,1]$.

  \smallskip
  \item[(iii)] For the Bernstein basis polynomials in $(0,1)$ we have

               \smallskip
               $P_{n,0}'(x)=-n(1-x)^{n-1}<0$,

               \smallskip
               $P_{n,k}'(x)=n\binom{n}{k}x^{k-1}(1-x)^{n-k-1}\big(\frac{k}{n}-x\big)$, \ and
               $P_{n,k}'(x)=0$ \ only if \ $x=\frac{k}{n}$,

               \smallskip
               $P_{n,n}'(x)=n x^{n-1}>0$.

  \smallskip
  \item[(iv)]  Now, from (ii) and (iii) for $x\in(0,1)$,

               \smallskip
               $-\varphi(x)T_{n,0}'(x)\,P_{n,0}'(x)>0$,

               \smallskip
               $-\varphi(x)T_{n,1}'(x)\,P_{n,1}'(x)>0$ for $x\in\big(\xi_1,\frac{1}{n}\big)=\big(0,\frac{1}{n}\big)$,

               \smallskip
               $-\varphi(x)T_{n,k}'(x)\,P_{n,k}'(x)>0$ for $x\in\big(\frac{k-1}{n},\xi_k\big)$, \
               $k=2,\ldots,\lfloor\frac{n-1}{2}\rfloor$,

               $-\varphi(x)T_{n,k}'(x)\,P_{n,k}'(x)<0$ for $x\in\big(\xi_k,\frac{k}{n}\big)$, \
               $k=2,\ldots,\lfloor\frac{n-1}{2}\rfloor$,

               \smallskip
               $-\varphi(x)T_{n,n}'(x)\,P_{n,n}'(x)>0$,

  \smallskip
  \item[(v)]   From the observations in (ii)--(iv) it follows that
               $$
                 -\varphi(x)T_{n,k}'(x)\,P_{n,k}'(x) > 0, \qquad k=0,\ldots,n
               $$
               except
               \begin{alignat*}{3}
                 & -\varphi(x)T_{n,k}'(x)\,P_{n,k}'(x) < 0, \quad && x\in \big(\xi_k,\tfrac{k}{n}\big),
                   && \quad k=1,\ldots,\big\lfloor\tfrac{n-1}{2}\big\rfloor, \\
                 & -\varphi(x)T_{n,n-k}'(x)\,P_{n,n-k}'(x) < 0, \quad && x\in \big(\tfrac{n-k}{n},\xi_{n-k}\big),
                   && \quad k=1,\ldots,\big\lfloor\tfrac{n-1}{2}\big\rfloor.
               \end{alignat*}
               Hence,
               $$
                 \sum_{k=0}^n \Big|\frac{-2\varphi(x)T_{n,k}'(x)}{n}\,P_{n,k}'(x)\Big| = S(x) = 4(n-1),
                 \quad x\in \big[0,\tfrac12\big]\setminus \textstyle{\bigcup\limits_{k=1}^{\lfloor\frac{n-1}{2}\rfloor}} \,(\xi_k,\tfrac{k}{n}).
               $$
               Therefore, for $k=1,\ldots,\big\lfloor\frac{n-1}{2}\big\rfloor$,
               \begin{equation} \label{eq:4.5}
                 b_n(x) =
                 \begin{cases}
                   4(n-1), & \ x\in \big[\frac{k-1}{n},\xi_k\big], \smallskip\\
                   4(n-1) + \dfrac{2\varphi(x)}{n}\big|T_{n,k}'(x) P_{n,k}'(x)\big|,  & \ x\in \big[\xi_k,\frac{k}{n}].
                 \end{cases}
               \end{equation}

              Moreover,
              $$
                b_n(x) = 4(n-1), \qquad x\in\big[\tfrac{n-2}{2n},\tfrac{n+2}{2n}\big], \ n \ \text{even},
                \ \text{and} \ x\in\big[\tfrac{n-1}{2n},\tfrac{n+1}{2n}\big], \ n \ \text{odd}.
              $$

  \smallskip
  \item[(vi)] This means that we have to estimate the maxima of the functions
              $$
                s_k(x) := \bigg|\frac{-2\varphi(x)T_{n,k}'(x)}{n}\,P_{n,k}'(x)\bigg|, \qquad
                x\in \big[\xi_k,\tfrac{k}{n}\big], \quad k=1,\ldots,\big\lfloor\tfrac{n-1}{2}\big\rfloor.
              $$
              By (iv) for $k=1$ we have:
              $$
                s_1(x) = \frac{-2\varphi(x)T_{n,1}'(x)}{n}\,P_{n,1}'(x)
                       = 2n(n-1)(n-2)x\Big(\frac{1}{n}-x\Big)(1-x)^{n-3}.
              $$
              Since
              $$
                \max_{x\in [0,1/n]} x\Big(\frac{1}{n}-x\Big) = \frac{1}{4n^2} \quad\text{and}\quad
                (1-x)^{n-3} \le 1,
              $$
              we obtain
              \begin{equation} \label{eq:4.6}
                h_1 := \max_{x\in[0,1/n]} s_1(x) \le \frac{2n(n-1)(n-2)}{4n^2}  \le \frac{n}{2}.
              \end{equation}

              Let us fix $k\in \big\{2,\ldots,\big\lfloor\tfrac{n-1}{2}\big\rfloor\big\}$. We
              estimate the local extremum
              $$
                h_k := \max_{x\in [\xi_k,k/n]} s_k(x).
              $$
              According to (iv) we have
              $$
                s_k(x) = \frac{2\varphi(x)}{n}\, T_{n,k}'(x) P_{n,k}'(x)
                       = \frac{2\varphi(x)}{n}\, T_{n,k}'(x) \binom{n}{k} x^{k-1} (1-x)^{n-k-1}\Big(\frac{k}{n}-x\Big),
              $$
              i.e.
              \begin{equation} \label{eq:4.7}
                s_k(x) = \frac{2}{n}\, T_{n,k}'(x) P_{n,k}(x) \Big(\frac{k}{n}-x\Big).
              \end{equation}
              Function $T_{n,k}'(x)$ is strictly increasing in $\big[\frac{k-1}{n},\frac{k}{n}\big]$ and
              change sign only at point $\xi_k=\frac{\sqrt{\binom{k}{2}}}{\sqrt{\binom{k}{2}}+\sqrt{\binom{n-k}{2}}}$.
              Then, for $x\in \big[\xi_k,\frac{k}{n}\big]$,
              $$
                \max_{x\in [\xi_k,k/n]} \,T_{n,k}'(x)
                  = T_{n,k}'\big(\tfrac{k}{n}\big)
                  = - \frac{k(k-1)n^2}{k^2} + \frac{(n-k)(n-k-1)n^2}{(n-k)^2}
                  = n^2 \Big(\frac{1}{k} - \frac{1}{n-k}\Big).
              $$
              The function $h(x)=\frac{1}{x}-\frac{1}{n-x}$ is decreasing in $\big(0,\frac{n}{2}\big)$
              since $h'(x)=\big(\frac{1}{x}-\frac{1}{n-x})'<0$, hence for
              $k\in \big\{2,\ldots,\big\lfloor\tfrac{n-1}{2}\big\rfloor\big\}$
              \begin{equation} \label{eq:4.8}
                T_{n,k}'(x) \le n^2 \Big(\frac{1}{k} - \frac{1}{n-k}\Big)
                            \le n^2 \Big(\frac12 - \frac{1}{n-2}\Big)
                            \le \frac{n^2}{2}.
              \end{equation}

              Also, $\frac{k-1}{n}\le \xi_k\le \frac{k}{n}$ and for $x\in \big[\xi_k,\frac{k}{n}\big]$
              we have $\frac{k}{n} - x \le \frac{1}{n}$. Since $0\le P_{n,k}(x)\le 1$ in $[0,1]$,
              it follows from \eqref{eq:4.7} and \eqref{eq:4.8} that
              $$
                h_k \le \frac{2}{n}\cdot\frac{n^2}{2}\cdot\frac{1}{n} \le 1.
              $$

              Taking into account \eqref{eq:4.6}, for $n\ge 2$ we have
              \begin{equation} \label{eq:4.9}
                h_k\le h_1 \le \frac{n}{2}, \qquad k=1,\ldots,\big\lfloor \frac{n-1}{2}\big\rfloor.
              \end{equation}
              Finally, for $b_n(x)$, using \eqref{eq:4.5} and \eqref{eq:4.9}, we obtain the estimate
              $$
                b_n(x) \le 4(n-1) + \max_{1\le k\le \big\lfloor \frac{n-1}{2}\big\rfloor} h_k
                       \le 4(n-1) + \frac{n}{2},
              $$
              or
              \begin{equation} \label{eq:4.10}
                b_n(x) \le 4.5\,n, \qquad x\in [0,1].
              \end{equation}
\end{enumerate}

\smallskip
\underline{\sl 3. Estimation of $c_n(x)$.} \ We apply Cauchy inequality and
Proposition~\ref{pr:2.3}\,(b) with $\alpha=0$ and $\alpha=1$:
\begin{align*}
  c_n(x) & = \sum_{k=0}^n \Big|T_{n,k}(x)\Big(1 - \frac{1}{n}\,T_{n,k}(x)\Big)\Big| P_{n,k}(x) \\
         & \le \sqrt{\sum_{k=0}^n T_{n,k}^2(x) P_{n,k}(x)} \sqrt{\sum_{k=0}^n \Big(1-\frac{1}{n}\,T_{n,k}(x)\Big)^2 P_{n,k}(x)} \\
         & = \sqrt{\Phi(0)n^2}\cdot\sqrt{\Phi(1)} = n\sqrt{2-\frac{2}{n}}\cdot\sqrt{3-\frac{2}{n}}.
\end{align*}
Then,
\begin{equation} \label{eq:4.11}
  c_n(x) \le \sqrt{6}\,n, \qquad x\in [0,1].
\end{equation}

From \eqref{eq:4.3}, \eqref{eq:4.10} and \eqref{eq:4.11} we obtain
$$
  \sum_{k=0}^n \big|\wD \wP_{n,k}(x)\big| \le a_n(x) + b_n(x) + c_n(x) \le 2(n-1) + 4.5\,n + \sqrt{6}\,n
  \le \big(6.5+\sqrt{6}\big)n.
$$
Therefore
$$
  \big\|\wD\wU_n f\big\| \le \wC n\|f\|, \qquad \wC := 6.5+\sqrt{6} .
$$
\end{proof}

\smallskip
Now we are ready to prove a strong converse inequality of Type B accordind to Ditzian-Ivanov
classification.

\smallskip
\begin{proof}[Proof of Theorem~\ref{th:1.2}]
We follow the approach of Ditzian and Ivanov~\cite{DiIv1993}.

Let $n\in\mathbb{N}$, $n\ge 2$, $f\in C[0,1]$ and $\lambda(n)$, $\theta(n)$ be defined as in
Proposition~\ref{pr:2.4}. From the Voronovskaya type inequality in Lemma~\ref{le:4.1} for the
operator $\wU_{\ell}$ and function $\wD^2 \wU_n^2 f$ instead of $f$ we have
\begin{align*}
  \lambda(\ell) \big\|\wD^2 \wU_n^3 f\big\|
    & = \big\|\lambda(\ell)\wD^2 \wU_n^3 f\big\| \\
    & = \big\|\wU_{\ell} \wU_n^3 f - \wU_n^3 f +\lambda(\ell)\wD^2 \wU_n^3 f
        - \wU_{\ell} \wU_n^3 f + \wU_n^3 f\big\| \\
    & \le \big\|\wU_{\ell} \wU_n^3 f - \wU_n^3 f +\lambda(\ell)\wD^2 \wU_n^3 f\big\|
          + \big\|\wU_{\ell} \wU_n^3 f - \wU_n^3 f\big\| \\
    & \le \theta(\ell) \big\|\wD^3 \wU_n^3 f\big\| + \big\|\wU_n^3\big(\wU_{\ell} f - f\big)\big\|.
\end{align*}
Now, using Lemma~\ref{le:4.2} for the function $\wD^2 \wU_n^2 f$ and in addition Lemma~\ref{le:3.1}
repeatedly three times we obtain
\begin{align*}
  \lambda(\ell) \big\|\wD^2 \wU_n^3 f\big\|
    & \le \wC\,n\,\theta(\ell) \big\|\wD^2 \wU_n^2 f\big\| + 3\sqrt{3}\,\big\|\wU_{\ell}f - f\big\| \\
    & = \wC\,n\,\theta(\ell) \big\|\wD^2 \wU_n^2 (f - \wU_n f) + \wD^2 \wU_n^3 f\big\|
        + 3\sqrt{3}\,\big\|\wU_{\ell}f - f\big\| \\
    & \le \wC\,n\,\theta(\ell) \big\|\wD^2 \wU_n^2 (f - \wU_n f)\big\|
          + \wC\,n\,\theta(\ell) \big\|\wD^2 \wU_n^3 f\big\| + 3\sqrt{3}\,\big\|\wU_{\ell}f - f\big\|.
\end{align*}
Applying the Bernstein type inequality Lemma~\ref{le:4.2} twice for $f-\wU_n f$ yields
$$
  \lambda(\ell) \big\|\wD^2 \wU_n^3 f\big\|
  \le \wC^3 n^3\theta(\ell) \big\|f - \wU_n f\big\|
      + 3\sqrt{3}\,\big\|\wU_{\ell} - f\big\| + \wC\,n\,\theta(\ell) \big\|\wD^2 \wU_n^3 f\big\|.
$$
From inequalities \eqref{eq:2.11} and \eqref{eq:2.12} of Proposition~\ref{pr:2.4} we get
$$
  \frac{1}{2\ell^2} \big\|\wD^2 \wU_n^3 f\big\|
  \le \frac{4\wC^3 n^3}{9\ell^3} \big\|f - \wU_n f\big\|
      + 3\sqrt{3}\,\big\|\wU_{\ell} - f\big\| + \frac{4\wC n}{9\ell^3} \big\|\wD^2 \wU_n^3 f\big\|.
$$
Let us choose $\ell$ sufficiently large such that
$$
  \frac{4\wC n}{9\ell^3} \le \frac{1}{4\ell^2}, \qquad\text{i.e.}\qquad \ell \ge \frac{16\wC}{9}\,n.
$$
If we set $L=\frac{16\wC}{9}$, for all integers $\ell\ge Ln$ we have
\begin{gather}
  \frac{1}{2\ell^2} \big\|\wD^2 \wU_n^3 f\big\|
    \le \frac{4\wC^3 n^3}{9\ell^3} \big\|f - \wU_n f\big\|
        + 3\sqrt{3}\,\big\|\wU_{\ell} - f\big\| + \frac{1}{4\ell^2} \big\|\wD^2 \wU_n^3 f\big\|, \notag\\
  \frac{1}{4\ell^2} \big\|\wD^2 \wU_n^3 f\big\|
    \le \frac{4\wC^3 n^3}{9\ell^3} \big\|f - \wU_n f\big\| + 3\sqrt{3}\,\big\|\wU_{\ell} - f\big\|, \notag\\
  \frac{1}{n^2} \big\|\wD^2 \wU_n^3 f\big\|
    \le \wC^2\,\big\|f - \wU_n f\big\|
        + 12\sqrt{3}\,\frac{\ell^2}{n^2}\,\big\|\wU_{\ell} - f\big\|. \label{eq:4.12}
\end{gather}
By using Lemma~\ref{le:3.1},
\begin{align*}
  \big\|f-\wU_n^3 f\big\|
  & \le \big\|f - \wU_n f\big\| + \big\|\wU_n f - \wU_n^2 f\big\| + \big\|\wU_n^2 f - \wU_n^3 f\big\| \\
  & \le (1 + \sqrt{3} + (\sqrt{3})^2)\big\|f - \wU_n f\big\|,
\end{align*}
and we obtain the inequality
\begin{equation} \label{eq:4.13}
  \big\|f - \wU_n^3 f\big\| \le (4+\sqrt{3})\big\|f - \wU_n f\big\|.
\end{equation}

It remains to complete the estimation of the K-functional. Since
$\wU_n^3 f\in W^2_0(\varphi)[0,1]$, from \eqref{eq:4.12} and \eqref{eq:4.13} it
follows
\begin{align*}
  K\Big(f,\frac{1}{n^2}\Big)
    & = \inf \Big\{\|f-g\| + \frac{1}{n^2}\,\big\|\wD^2 g\big\|: \ g\in W^2_0(\varphi)[0,1],\,\wD g\in W^2(\varphi)[0,1] \Big\} \\
    & \le \big\|f-\wU_n^3 f\| + \frac{1}{n^2}\,\big\|\wD^2 \wU_n^3 f\big\| \\
    & \le \Big(4 + \sqrt{3} + \wC^2\Big)\big\|\wU_n f-f\big\|
          + 12\sqrt{3}\,\frac{\ell^2}{n^2}\,\big\|\wU_{\ell} f-f\big\|.
\end{align*}

Therefore,
$$
  K\Big(f,\frac{1}{n^2}\Big) \le
  C\,\frac{\ell^2}{n^2}\big(\big\|\wU_n f - f \big\| + \big\|\wU_{\ell} f - f) \big\| \big)
$$
for all $\ell\ge Ln$, where $C=4+\sqrt{3}+\wC^2$ and $L=\frac{16\wC}{9}$, $\wC=6.5+\sqrt{6}$.
\end{proof}

%====================================================================================================

\end{document}